\documentclass{amsart}
\usepackage{amsmath}
\usepackage{amssymb}
\usepackage{hyperref}
\usepackage{enumitem}
\usepackage{mathrsfs}
\usepackage{rank-2-roots}
\usepackage{tikz}

\numberwithin{equation}{section}

\newcommand{\R}[0]{\mathbb{R}}

\newtheorem{thm}{Theorem}[section]

\newtheorem{prop}[thm]{Proposition}
\newtheorem{lem}[thm]{Lemma}
\newtheorem{conj}[thm]{Conjecture}

\newtheorem*{defn*}{Definition}

\newtheorem{rem}[thm]{Remark}

\numberwithin{equation}{section}

\usepackage{etoolbox}

\makeatletter
\patchcmd{\@settitle}{\uppercasenonmath\@title}{}{}{}
\patchcmd{\@setauthors}{\MakeUppercase}{}{}{}
\patchcmd{\section}{\scshape}{}{}{}
\makeatother

\makeatletter
\@namedef{subjclassname@2020}{\textup{2020} Mathematics Subject Classification}
\makeatother

\title[Pointwise character bounds for $\mathrm{SU}(3)$]{Pointwise character bounds for $\mathrm{SU}(3)$}

\author[Y. Zhang]{Yunfeng Zhang}

\address{Department of Mathematics, University of Mississippi,
University, MS 38677}
\email{yzhang19@olemiss.edu}


\begin{document}

\begin{abstract}
We present a basic pointwise bound for the irreducible characters of $\mathrm{SU}(3)$ and, as an application, derive new $L^p$ bounds for these characters. Our approach is based on the descent of characters to singular sets and the cancellation in this formula. 
\end{abstract}

\maketitle

\section{Introduction}
Let $\mathrm{SU}(n)$ denote the special unitary group of degree $n$. 
We have the following basic pointwise bound for the irreducible characters of $\mathrm{SU}(2)$:
\begin{align}\label{eq: rank one}
    \left|\frac{\sin ((n+1)\theta)}{\sin \theta}\right|\leq\min\{n+1,|\sin\theta|^{-1}\},\qquad n\in\mathbb{Z}_{\geq 0}, \ \theta\in\R.
\end{align}
The goal of this paper is to prove an analogous pointwise bound for the irreducible characters of $\mathrm{SU}(3)$ and, as an application, derive new $L^p$ bounds for these characters.

Bounds for characters, both pointwise and in $L^p$ norms, are central to harmonic analysis on compact groups. For example, they arise in the study of summability of Fourier series \cite{Cle76, ST78, GST82, Giu83}, lacunary sets \cite{Cec72, Doo79, Giu83}, and the singularity of central measures \cite{Har98, HWY00}. As characters form a distinguished class of spherical functions on compact symmetric spaces, their analytic behavior also serves as a model for more general spherical functions \cite{Cle88, Mar16}. Finally, since characters are eigenfunctions of the Laplace operator on compact Lie groups, their $L^p$ bounds shed light on the general problem of Laplace eigenfunction bounds on manifolds \cite{Sog93, Zel17}.

Throughout the paper, let $G$ denote $\mathrm{SU}(3)$. Let $T$ be a maximal torus of $G$. Let $\mathfrak{t}$ be the Lie algebra of $T$, and $\mathfrak{t}^*$ be the real dual of $\mathfrak{t}$. Let $\Phi^+\subset i\mathfrak{t}^*$ be a positive root system for $G$. 
For example, we may pick $\Phi^+=\{\alpha_1,\alpha_2,-\alpha_0\}$; see Figure \ref{fig:roots}. Let $\Lambda^+\subset i\mathfrak{t}^*$ be the set of dominant weights. Let $W$ be the Weyl group of $G$. For $\mu\in\Lambda^+$, let $\chi_\mu:G\to\mathbb{C}$ be the corresponding irreducible character of $G$. 
Let
$$\chi(\mu,H):=\chi_\mu(\exp H),\qquad \mu\in\Lambda^+,\ H\in\mathfrak{t}.$$
Let $\rho$ be the half sum of positive roots. We state our main theorem.

\begin{thm}\label{thm: main}
For $\mu\in\Lambda^+$ and $H\in\mathfrak{t}$, there exists a uniform constant $C>0$ such that  
\begin{align}\label{eq: first}
    |\chi(\mu, H)|\leq C\sum_{s\in W}\prod_{\alpha\in\Phi^+}\min\left\{|\langle s(\mu+\rho),\alpha\rangle|, \left|\sin\left(\frac{\alpha( H)}{2 i}\right)\right|^{-1}\right\}.
\end{align}
\end{thm}

For $\mu\in\Lambda^+$, let $d_\mu=\chi(\mu,0)$ be the dimension of the associated irreducible representation. Using the Weyl dimension formula \cite{Wey39} 
\begin{align}\label{eq: dimension}
d_\mu=\prod_{\alpha\in\Phi^+}\frac{\langle  \mu+\rho,\alpha\rangle}{\langle  \rho,\alpha\rangle},
\end{align}
we may rewrite the inequality \eqref{eq: first} as 
\begin{align}\label{eq: second}
    |\chi(\mu, H)|\leq C d_\mu\sum_{s\in W}\prod_{\alpha\in\Phi^+}\left(1+ \left|\langle s(\mu+\rho),\alpha\rangle\sin\left(\frac{\alpha( H)}{2 i}\right)\right|\right)^{-1}.
\end{align}

The above inequality suggests a natural extension to spherical functions on symmetric spaces of compact type: 

\begin{conj}\label{conj}
Let $\psi(\mu,H)$ be the spherical function on an irreducible symmetric space $X$ of compact type, with spectral parameter $\mu\in \Lambda^+\subset i\mathfrak{t}^*$ and spatial parameter $H\in \mathfrak{t}$, normalized such that $\psi(\mu,0)=1$. 
Let $W$ denote the Weyl group, and let $\Phi^+$ be the positive restricted root system for $X$, with the root multiplicity function denoted by $\Phi^+\ni\alpha\mapsto m(\alpha)\in\mathbb{Z}_{\geq 1}$.  
Then 
\begin{align} \label{eq: conj}
    |\psi(\mu, H)|\leq C\sum_{s\in W}\prod_{\alpha\in\Phi^+}\left(1+ \left|\langle s\mu,\alpha\rangle\sin\left(i\alpha( H)\right)\right|\right)^{-\frac{m(\alpha)}{2}}
\end{align}
holds for some uniform constant $C$. 
\end{conj}

The bound \eqref{eq: conj} is naturally expected from the heuristics of stationary phase. Write the symmetric space $X$ as the standard quotient $U/K$, where  
$U$ is the maximal compact subgroup of a complex semisimple Lie group $U_\mathbb{C}$. 
Clerc gave an integral formula for the spherical function (Theorem 2.2 of \cite{Cle88}), namely,
$$\psi(\mu,H)= \int_{K_{\exp H}} e^{\mu(\mathcal{H}((\exp H) k))}\ dk,$$
where $\mathcal{H}$ denotes the complexified Iwasawa projection which is multi-valued and holomorphic on an open dense subset $\omega$ of $U_\mathbb{C}$, 
and $K_{\exp H}=\{k\in K: (\exp H) k\in\omega\}$ is an open dense subset of $K$. 
Assuming that both the spectral parameter $\mu$ and the spatial parameter $H$ are regular, 
Clerc found that the critical points of the phase function $k\mapsto \mu(\mathcal{H}((\exp H) k))$ are indexed by the Weyl group $W$, and for each $s\in W$, the eigenvalues of the corresponding Hessian have absolute values exactly 
$\left|\langle s\mu,\alpha\rangle\sin\left(i\alpha( H)\right)\right|$, with $\alpha\in\Phi^+$ counted with their multiplicities. Thus \eqref{eq: conj} is the expected bound from stationary phase. However, proving \eqref{eq: conj} via stationary phase would require establishing the expected normal forms of the phase function that are locally uniform in both $\mu$ and $H$ which can be singular, and this seems to be a formidable task for now \cite{DKV83, BP16, Mar16, Li25}. 

What was known about Conjecture \ref{conj}? 
Assume that $\mu$ remains in a cone away from the Weyl chamber walls---more precisely, assume that there exists a constant $c>0$ with $|\langle\mu,\alpha\rangle|> c|\mu|$ for all $\alpha\in\Phi^+$. Then the bound \eqref{eq: conj} reduces to 
\begin{align} \label{eq: conj regular}
    |\psi(\mu, H)|\leq C\prod_{\alpha\in\Phi^+}\left(1+ |\mu|\cdot \left|\sin\left(i\alpha( H)\right)\right|\right)^{-\frac{m(\alpha)}{2}}.
\end{align}
Using stationary phase, 
Clerc established \eqref{eq: conj regular} for $H$ lying in a compact set contained in the open set of regular points (Theorem 3.4 of \cite{Cle88}), while Marshall established \eqref{eq: conj regular}  for $H$ lying in a small neighborhood of the origin (Theorem 1.6 of \cite{Mar16}). 
Our 
Theorem \ref{thm: main} seems to be the first result establishing the full pointwise bound \eqref{eq: conj} for spherical functions on a higher-rank compact symmetric space. 

Our proof relies on explicit character formulas, namely the Weyl character formula and its consequences for the descent of characters to singular sets \eqref{eq: cha}. The Weyl group symmetry of the characters plays an essential role in the proof. 
The parabolic subgroups of the Weyl group and their cosets dictate the descent of characters to singular sets \eqref{eq: cha}. A key observation is that an appropriate grouping of terms in the descent formula, indexed by cosets, produces the cancellation needed to obtain the desired bounds; see \eqref{eq: cancellation case 2} and \eqref{eq: cancellation case 3}. Additionally, the Weyl group symmetry of both the characters and the bound \eqref{eq: main bd}, in both $\mu$ and $H$, allows one to reduce the proof to representative configurations, thereby simplifying the analysis. The framework developed here---based on the Weyl group symmetry, parabolic subgroups, and cancellation in the descent formula---is expected to extend to more general compact Lie groups, and we defer the treatment of the general case to future work.

We now discuss several consequences of Theorem \ref{thm: main}. 
For $\mu\in\Lambda^+$, let $\overline{\mu}$ and $\underline{\mu}$ denote the maximum and minimum of 
$\{\langle \mu+\rho,\alpha\rangle : \alpha \in \Phi^+\}$, respectively. Note that the second largest element of 
$\{\langle \mu+\rho,\alpha\rangle: \alpha \in \Phi^+\}$ 
is always comparable to $\overline{\mu}$ (in fact, it is at least $\tfrac{1}{2}\overline{\mu}$). It follows from the Weyl dimension formula \eqref{eq: dimension} that
\[
d_\mu \approx \overline{\mu}^2\underline{\mu}.
\]
For $H\in\mathfrak{t}$, let $c(H)$ denote 
the product of the two largest elements of 
$\{|\sin(\alpha(H)/2i)|:\alpha\in\Phi^+\}$. Then $\exp H$ is non-central in $G$ if and only if $c(H)\neq 0$.

As an immediate corollary of \eqref{eq: second}, we have:

\begin{thm}
For $\mu\in\Lambda^+$ and $H\in\mathfrak{t}$ such that $\exp H$ is non-central in $G$, we have 
\begin{align}\label{eq: pointwise}
    \frac{|\chi(\mu,H)|}{d_\mu}
    \leq C c(H)^{-1}(\overline{\mu}\underline{\mu})^{-1}.
\end{align}
\end{thm}
\begin{proof}
Suppose the two largest elements of 
$\{|\sin(\alpha(H)/2i)|:\alpha\in\Phi^+\}$ 
are attained at the positive roots $\alpha^1,\alpha^2$. Then 
\eqref{eq: second} implies 
\begin{align*}
    \frac{|\chi(\mu,H)|}{d_\mu}
    \leq C\sum_{s\in W} c(H)^{-1}(|\langle s(\mu+\rho),\alpha^1\rangle|\cdot|\langle s(\mu+\rho),\alpha^2\rangle|)^{-1}.
\end{align*}
As $\{|\langle s(\mu+\rho),\alpha\rangle|:\alpha\in\Phi^+\}
=
\{\langle \mu+\rho,\alpha\rangle:\alpha\in\Phi^+\}$, at least one of
$
|\langle s(\mu+\rho),\alpha^1\rangle|$ and  $|\langle s(\mu+\rho),\alpha^2\rangle|$ 
is comparable to $\overline{\mu}$, while the other is bounded below by $\underline{\mu}$. Hence
\(
|\langle s(\mu+\rho),\alpha^1\rangle|\cdot
|\langle s(\mu+\rho),\alpha^2\rangle|
\geq
\overline{\mu}\,\underline{\mu}
\). The result follows.  
\end{proof}

Bounds such as \eqref{eq: pointwise} have been valuable in understanding the singularity of central measures on compact Lie groups \cite{Har98, HWY00}. In particular, it was shown in Theorem 1.1 of \cite{HWY00} that 
\begin{align}\label{eq: pointwise old}
    \frac{|\chi(\mu,H)|}{d_\mu}
    \leq C c(H)^{-1}d_{\mu}^{-1/2},
\end{align}
which was used to characterize $L^2$ powers of orbital measures. 
Noting that \(d_\mu \approx \overline{\mu}^2 \underline{\mu}\), our bound \eqref{eq: pointwise} improves upon \eqref{eq: pointwise old} by a factor of \(\underline{\mu}^{-1/2}\). 

As another application of Theorem \ref{thm: main}, we prove the following $L^p$ bound for irreducible characters of $\mathrm{SU}(3)$.

\begin{thm}\label{thm: Lp}  
For $\mu\in\Lambda^+$, we have 
    \begin{align}\label{eq: singular bound}
\|\chi_\mu\|_{L^p(G)}
\leq C\cdot
\begin{cases}
1, & 0<p<\frac83,\\[2mm]
\log^{3/8}(2+\underline{\mu}), & p=\frac83,\\[2mm]
\underline{\mu}^{\,3-\frac{8}{p}}, & \frac83<p<3,\\[2mm]
\underline{\mu}^{1/3}\log^{1/3}(2+\overline{\mu}/\underline{\mu}), & p=3,\\[2mm]
\overline{\mu}^{\,1-\frac{3}{p}}\underline{\mu}^{\,2-\frac{5}{p}}, & 3<p<5,\\[2mm]
\overline{\mu}^{2/5}\underline{\mu}\,\log^{1/5}(2+\overline{\mu}/\underline{\mu}), & p=5,\\[2mm]
\overline{\mu}^{\,2-\frac{8}{p}}\underline{\mu}, & p>5.
\end{cases}
\end{align}
\end{thm}

In \cite{Zha26}, we established
    \begin{align}\label{eq: regular bound}
        \|\chi_\mu\|_{L^p(G)}
\leq C\cdot
\begin{cases}
1, & 2\leq p<\frac83,\\[2mm]
\log^{3/8}(2+\overline{\mu}), & p=\frac83,\\[2mm]
\overline{\mu}^{\,3-\frac{8}{p}}, & p>\frac83.
\end{cases}
    \end{align}

In the ``regular regime'' $\overline{\mu} \approx \underline{\mu}$, the new bound \eqref{eq: singular bound} matches the bound \eqref{eq: regular bound}, which is expected to be sharp in this regime and was verified to be so for $\mu = N\rho$ ($N \in \mathbb{Z}_{\ge 1}$) \cite{Zha26}. 
However, in the ``singular regime'' $\overline{\mu}\gg\underline{\mu}$, an inspection shows that \eqref{eq: singular bound} improves upon \eqref{eq: regular bound}.
We expect \eqref{eq: singular bound} to be sharp for all spectral parameters $\mu$ and all $p>0$. 

Using $d_\mu\approx \overline{\mu}^2\underline{\mu}$, one readily obtains the following corollary of Theorem~\ref{thm: Lp}:

\begin{thm}
    For $\mu\in\Lambda^+$, we have 
    \begin{align}\label{eq: Lp dimension}
\|\chi_\mu\|_{L^p(G)}
\leq C
d_{\mu}^{\,1-\frac{8}{3p}},\qquad p>\frac83.
\end{align}
\end{thm}

The above bound improves upon Corollary 2.4 of \cite{Doo79}, which established \eqref{eq: Lp dimension} up to an arbitrarily small loss in the power of \(d_\mu\).

Finally, we note another interesting consequence of Theorem~\ref{thm: Lp}: 

\begin{thm}
    For any sequence \(\mu_n \in \Lambda^+\) with \(\overline{\mu_n} \to \infty\) and 
\(\underline{\mu_n}\leq C\), we have \(d_{\mu_n} \to \infty\) while 
\(\|\chi_{\mu_n}\|_{L^p(G)} \le K(p,C) < \infty\) for all \(n\) and all \(p < 3\). 
\end{thm}

This recovers a result of \cite{Rid72}, where such sequences were shown to exist for $\mathrm{SU}(n)$ and later generalized to all classical compact Lie groups in \cite{Giu83}, with applications to lacunary sets. In contrast, our result applies uniformly to all sequences in $\Lambda^+$ satisfying the conditions in the above theorem, going beyond the particular constructions in \cite{Rid72, Giu83}.

\begin{rem}
There is a parallel theory for spherical functions on symmetric spaces of noncompact type. In fact, our method of proving Theorem \ref{thm: main} applies equally well to establishing a similar pointwise bound for spherical functions on $\mathrm{SL}(3,\mathbb{C})/\mathrm{SU}(3)$, which is the noncompact dual of $\mathrm{SU}(3)$. Namely, let $\varphi(\lambda,H)$ be the spherical function on $\mathrm{SL}(3,\mathbb{C})/\mathrm{SU}(3)$ with spectral parameter $\lambda\in \mathfrak{a}^*$ and spatial parameter $H\in \mathfrak{a}$, normalized such that $\varphi(\lambda,0)=1$. 
Let 
$\varphi_0(H)=\varphi(0,H)$. 
Then 
\begin{align}\label{eq: complex}
    |\varphi(\lambda, H)|\leq C\varphi_0(H)\sum_{s\in W}\prod_{\alpha\in\Phi^+}\left( 1+
    |\langle s\lambda, \alpha\rangle \langle \alpha,H\rangle|\right)^{-1}
\end{align}
holds for some uniform constant $C$. We note that in a recent preprint \cite{BMMP26}, the same estimate \eqref{eq: complex} was established for spherical functions on general complex semisimple groups. The authors employ methods different from ours, based on the Gangolli--Varadarajan asymptotics for spherical functions on noncompact symmetric spaces, which do not appear to be adaptable to compact symmetric spaces.
\end{rem}

\begin{rem}
It might be interesting to compare our bounds \eqref{eq: second} and \eqref{eq: pointwise} with character bounds for finite groups. In \cite{LT24}, it was established that there exists an absolute constant $c>0$ such that for all finite quasisimple groups $G$ of Lie type, irreducible characters $\chi$ of $G$, and elements $g\in G$, it holds 
$$|\chi(g)|\leq \chi(1)^{1-c\frac{\log |g^G|}{\log |G|}}.$$
Here $g^G$ denotes the conjugacy class of $g$. All these bounds reflect the principle that elements with larger conjugacy classes (i.e., less singular elements) exhibit stronger oscillation and consequently smaller character values. 
\end{rem}

We overview the remaining sections of the paper. 
In Section \ref{sec: pre}, we set up the descent formula for characters (Lemma \ref{lem: cha}), use the Weyl group symmetry to reduce the bound \eqref{eq: first} to three cases of configurations of $\mu$ and $H$, and prove the first case. In Sections \ref{sec: case 2} and \ref{sec: case 3}, we exploit cancellation in the descent formula to treat the second and third cases, respectively. In Section \ref{sec: Lp}, we use Theorem \ref{thm: main} to deduce Theorem \ref{thm: Lp}.

\subsection*{Notations} We use $A\lesssim B$ to mean that there exists some constant $C>0$ such that $A\leq CB$. $A\approx B$ means that $A\lesssim B$ and $B\lesssim A$ hold at the same time. $A\ll B$ means that there exists a sufficiently small constant $c>0$ such that $A\leq cB$. 

\subsection*{Acknowledgments} The author thanks Xiaocheng Li and Simon Marshall for helpful discussions. The author thanks the referee for providing valuable corrections and suggestions that improved the paper. 

\section{Preliminaries}\label{sec: pre}
We use the invariant inner product on the Cartan subalgebra $\mathfrak{t}$ to identify
\[
i\mathfrak t^* \cong \mathfrak t.
\]
More precisely, for each $\alpha \in i\mathfrak t^*$, there exists a unique element
$H_\alpha \in \mathfrak t$ such that
\[
\alpha(H) = i\,\langle H_\alpha, H\rangle,
\qquad   H \in \mathfrak t.
\]
Via this identification, we will henceforth suppress the distinction between $\alpha$ and $H_\alpha$, regarding $\alpha\in i\mathfrak{t}^*$ simply as an element of $\mathfrak{t}$. 
Thus we may write
\[
\alpha(H) = i\,\langle \alpha, H\rangle,
\qquad H \in \mathfrak t.
\]
Here $\alpha$ may, in particular, be a root or a weight.
The Introduction retains the original notation and is not subject to the above convention. 

Let $\Phi=\{\pm\alpha_0,\pm\alpha_1,\pm\alpha_2\}\subset i\mathfrak{t}^*\cong\mathfrak{t}\cong\mathbb{R}^2$ be the root system for $\mathrm{SU}(3)$, and let $\widetilde{\Phi}=\{\alpha_0,\alpha_1,\alpha_2\}$ be the collection of extended simple roots for $\mathrm{SU}(3)$; see Figure \ref{fig:roots}. 
For $\alpha\in\widetilde{\Phi}$, let $s_{\alpha}:\mathfrak{t}\to\mathfrak{t}$ denote the linear mapping of reflection across the line through the origin orthogonal to $\alpha$. 
Let $W$ denote the Weyl group of $\mathrm{SU}(3)$. 
Explicitly, 
\begin{align}\label{eq: Weyl group}
    W=\{e, s_{\alpha_0}, s_{\alpha_1}, s_{\alpha_2}, s_{\alpha_{1}}s_{\alpha_{0}}, s_{\alpha_{0}}s_{\alpha_{1}}\},
\end{align}
where $e$ denotes the identity element of $W$. 

Let 
$$\Lambda=\left\{\mu\in\mathfrak{t}: \frac{2\langle\mu,\alpha\rangle}{\langle\alpha,\alpha\rangle}\in\mathbb{Z}\text{ for all }\alpha\in\widetilde{\Phi}\right\}$$
be the weight lattice. 
Let 
$$\Lambda_{\text{reg}}=\left\{\mu\in\mathfrak{t}: \frac{2\langle\mu,\alpha\rangle}{\langle\alpha,\alpha\rangle}\in\mathbb{Z}_{\neq 0}\text{ for }\alpha\in\widetilde{\Phi}\right\}$$ be the set of regular integral weights. 

Let $\Phi^+\subset\Phi$ be a positive root system; for example, we may take 
$\Phi^+=\{\alpha_1,\alpha_2,-\alpha_0\}$. Let $\rho=\frac{1}{2}\sum_{\alpha\in\Phi^+}\alpha$. 
Let 
$$\Lambda^+=\left\{\mu\in\mathfrak{t}: \frac{2\langle\mu,\alpha\rangle}{\langle\alpha,\alpha\rangle}\in\mathbb{Z}_{\geq 0}\text{ for }\alpha\in\Phi^+\right\}$$
be the set of dominant integral weights. Then we have 
$$\bigsqcup_{s\in W}s(\Lambda^++\rho)=\Lambda_{\text{reg}}.$$
To smooth the application of the Weyl group symmetry of the characters, 
we write
$$\widetilde{\chi}(\mu,H)= \chi(\mu-\rho,H).$$
Then the Weyl character formula reads \cite{Wey39}
    $$\widetilde{\chi}(\mu, H)=\frac{\sum_{s\in W} (\det s) e^{i\langle s\mu,H\rangle} }{\sum_{s\in W} (\det s) e^{i\langle s\rho,H\rangle} }, \qquad\mu\in\Lambda_{\text{reg}}, \ H\in\mathfrak{t}.$$
    The denominator may be factored so that we have 
\begin{align}\label{eq: Weyl}
   \widetilde{\chi}(\mu, H)=-\frac{\sum_{s\in W} (\det s) e^{i\langle s\mu,H\rangle} }{
    \prod_{\alpha\in\widetilde{\Phi}}2i\sin\left(\frac{\langle \alpha,H\rangle}{2}\right)
    }.
\end{align}
In particular, the negative sign in the right-hand side of the above equation is due to the fact that $\widetilde{\Phi}$ contains exactly one negative root, namely $\alpha_0$.

To make notation lighter, for all $\alpha$ in the root system $\Phi$ and $H\in\mathfrak{t}$, we denote 
    $$\|\langle\alpha,H\rangle\|:=\left|\sin\left(\frac{\langle\alpha, H\rangle}{2 }\right)\right|.$$
We may rewrite the bound \eqref{eq: first}  as 
\begin{align}\label{eq: main bd}
    |\widetilde{\chi}(\mu, H)|\lesssim\sum_{s\in W}\prod_{\alpha\in\widetilde{\Phi}}\min\left\{|\langle s\mu,\alpha\rangle|,   \|\langle\alpha,H\rangle\|^{-1}\right\}.
\end{align}

We will need the following lemma, which expresses the Weyl group as the product of a parabolic subgroup and a set of coset representatives. 

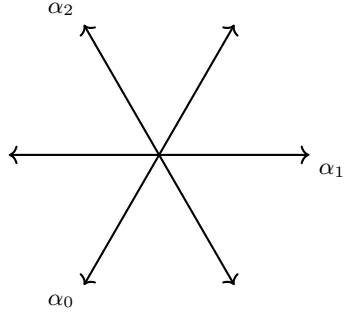
\begin{figure}
    \centering
\begin{tikzpicture}[scale=2, every node/.style={font=\small}]

\coordinate (a1) at (1,0);
\coordinate (a2) at (-0.5,0.8660254);
\coordinate (a0) at (-0.5,-0.8660254);

\coordinate (ma1) at (-1,0);
\coordinate (ma2) at (0.5,-0.8660254);
\coordinate (na0) at (0.5,0.8660254);

\draw[thick,->] (0,0) -- (a1);
\draw[thick,->] (0,0) -- (a2);
\draw[thick,->] (0,0) -- (a0);
\draw[thick,->] (0,0) -- (ma1);
\draw[thick,->] (0,0) -- (ma2);
\draw[thick,->] (0,0) -- (na0);

\node[below right] at (a1) {$\alpha_1$};
\node[above left] at (a2) {$\alpha_2$};
\node[below left] at (a0) {$\alpha_0$};

\end{tikzpicture}

    \caption{Root system for $\mathrm{SU}(3)$}
    \label{fig:roots}
\end{figure}

\begin{lem}[parabolic subgroups of the Weyl group and their cosets]\label{lem: parabolic}
For $j=0,1,2$, let $W^j=\{e,s_{\alpha_j}\}$, and 
let $W_j=\{e, s_{\alpha_{j+1}},s_{\alpha_{j+1}}s_{\alpha_{j}}\}$, with $\alpha_{3}$ understood to be $\alpha_0$. 
Then $W=W^jW_j$ for $j=0,1,2$.
\end{lem}
\begin{proof}
Using \eqref{eq: Weyl group}, this lemma can certainly be checked easily by hand. We refer to Section 1.10 of \cite{Hum90} for a general statement. 
\end{proof}

For $j=0,1,2$ and $X\in\mathfrak{t}$, 
let 
$$X^j=\mathrm{Proj}_{\alpha_j} X = \frac{\langle X, \alpha_j \rangle}{|\alpha_j|^2} \alpha_j,$$
$$X_j=X-X^j.$$
Then clearly from the definition of the weight lattice $\Lambda$ we have: 
\begin{lem}[weight projection]\label{lem: proj weights}
    For all $\mu\in \Lambda$ and $j=0,1,2$, $\mu^j\in \mathbb{Z}\cdot\frac12\alpha_j$. 
\end{lem}

For $j=0,1,2$ and $X,\nu\in \mathbb{R}\alpha_j$, let 
$$\widetilde{\chi}^j(\nu,X)=\frac{\sin(\langle\nu, X\rangle)}{\sin(\frac{\langle \alpha_j, X\rangle}{2})}.$$
We have the following key ``descent formula'' (in the terminology of \cite{CN01} for spherical functions on complex groups), which decomposes irreducible characters into sums of lower-rank characters, thereby clarifying their behavior near the alcove walls $\{H\in\mathfrak{t}:\langle\alpha,H\rangle=2\pi k\}$, $\alpha\in\widetilde{\Phi}$, $k\in\mathbb{Z}$. In the proof of Theorem \ref{thm: main}, a key observation is that this decomposition exhibits subtle cancellation among certain terms (see \eqref{eq: cancellation case 2} and \eqref{eq: cancellation case 3}).

\begin{lem}[character descent]\label{lem: cha}
For $j=0,1,2$, we have 
\begin{align}\label{eq: cha}
    \widetilde{\chi}(\mu, H)=\frac{1}{4\prod_{k\neq j}\sin\left(\frac{\langle\alpha_k, H\rangle}{2}\right)}\sum_{s_j\in W_j} (\det s_j) 
e^{i\langle s_j\mu,H_{j}\rangle}\widetilde{\chi}^j((s_j\mu)^j, H^j).
\end{align}
\end{lem}
\begin{proof}
A general version for compact Lie groups was established in Lemma 2.4 of \cite{Zha26}. We include the proof here for completeness.
Applying Lemma \ref{lem: parabolic}, we may rewrite the Weyl character formula \eqref{eq: Weyl} as 
\begin{align*}
\widetilde{\chi}(\mu, H)=-\frac{1}{
    \prod_{k\neq j}2i\sin\left(\frac{\langle \alpha_k,H\rangle}{2}\right)
    }\sum_{s_j\in W_j} (\det s_j)\frac{e^{i\langle s_j\mu,H\rangle}-e^{i\langle s_{\alpha_j}s_j\mu,H\rangle}}{2i\sin\left(\frac{\langle \alpha_j,H\rangle}{2}\right)}.
\end{align*}
Writing $H=H_j+H^j$ and noting that $s_{\alpha_j}H_j=H_j$ while $s_{\alpha_j}H^j=-H^j$, the above formula becomes 
\begin{align*}
    -\frac{1}{
    \prod_{k\neq j}2i\sin\left(\frac{\langle \alpha_k,H\rangle}{2}\right)
    }\sum_{s_j\in W_j} (\det s_j)
    e^{i\langle s_j\mu,H_j\rangle}
    \frac{e^{i\langle s_j\mu,H^j\rangle}-e^{-i\langle s_j\mu,H^j\rangle}}{2i\sin\left(\frac{\langle \alpha_j,H\rangle}{2}\right)}.
\end{align*}
Noting that 
\begin{align*}
    \frac{e^{i\langle s_j\mu,H^j\rangle}-e^{-i\langle s_j\mu,H^j\rangle}}{2i\sin\left(\frac{\langle \alpha_j,H\rangle}{2}\right)}=\frac{e^{i\langle (s_j\mu)^j,H^j\rangle}-e^{-i\langle (s_j\mu)^j,H^j\rangle}}{2i\sin\left(\frac{\langle \alpha_j,H\rangle}{2}\right)} =\widetilde{\chi}^j((s_j\mu)^j, H^j),
\end{align*}
and simplifying $-\frac{1}{(2i)^2}$ to $\frac14$, 
the proof is finished. 
\end{proof}

We need the following character bound for $\mathrm{SU}(2)$: 

\begin{lem}[rank-one bound]\label{lem: simple cha bd}
    For $j=0,1,2$, $X\in\mathbb{R} \alpha_j$, and $\nu\in \mathbb{Z}\cdot\frac12\alpha_j$, 
$$|\widetilde{\chi}^j (\nu, X)|\leq  \min\{2|\nu|/|\alpha_j|,\|\alpha_j(X)\|^{-1}\}.$$
\end{lem}
\begin{proof}
    This is the same as \eqref{eq: rank one}, the proof of which is elementary and omitted. 
\end{proof}

Combining Lemmas \ref{lem: proj weights} and \ref{lem: simple cha bd}, we obtain 
\begin{align}\label{eq: rank one bd}
    |\widetilde{\chi}^j (\mu^j, X)|\leq  \min\{2|\mu^j|/|\alpha_j|,\|\alpha_j(X)\|^{-1}\}, \qquad\text{for all }\mu\in\Lambda.
\end{align}

We now reduce the desired bound \eqref{eq: main bd}  to suitable cases for the spectral and spatial parameters, by using the Weyl group symmetry of this bound. 

First, we have the well-known Weyl group symmetry of the characters $$\widetilde{\chi}(\mu,H)=\widetilde{\chi}(\mu,sH), \qquad \text{for all }s\in W.$$
Thus both sides of the inequality \eqref{eq: main bd} are invariant under the symmetry 
$$H\mapsto sH,\qquad s\in W.$$
Noting that $W$ acts by permuting the projective roots (i.e., roots modulo sign), 
in proving \eqref{eq: main bd} we may assume without loss of generality that 
\begin{align}\label{eq: assumption on H}
   \|\langle\alpha_0, H\rangle\|\leq \|\langle\alpha_2, H\rangle\|\leq \|\langle\alpha_1, H\rangle\|. 
\end{align}

Second, it is also well known that 
$$\widetilde{\chi}(\mu,H)=(\det s)\widetilde{\chi}(s\mu,H), \qquad \text{for all }s\in W.$$
Thus both sides of the inequality \eqref{eq: main bd} are also invariant under the symmetry 
$$\mu\mapsto s\mu,\qquad s\in W.$$
Then we may also assume without loss of generality that 
\begin{align}\label{eq: assumption on mu}
     | \langle\mu,\alpha_1\rangle|\leq  |\langle\mu,\alpha_2\rangle|\leq |\langle\mu,\alpha_0\rangle|.   
\end{align}
This implies that 
\begin{align}\label{eq: assumption on mu 2}
    | \langle\mu,\alpha_1\rangle|\lesssim |\mu|\approx |\langle\mu,\alpha_2\rangle|\approx |\langle\mu,\alpha_0\rangle|.
\end{align}

\begin{proof}[Proof of Theorem \ref{thm: main}]
Note that if $\|\langle\alpha,H\rangle\|\leq |\mu|^{-1}$ for two roots $\alpha\in \widetilde{\Phi}$, then $\|\langle\alpha,H\rangle\|\lesssim |\mu|^{-1}$ for all $\alpha\in\widetilde{\Phi}$. To prove Theorem \ref{thm: main}, having assumed \eqref{eq: assumption on H} and \eqref{eq: assumption on mu} without loss of generality, it suffices to prove \eqref{eq: main bd} in the following three cases:

\noindent \textbf{Case 1}.  $\|\langle\alpha, H\rangle\|\lesssim |\mu|^{-1}$ for all $\alpha\in\widetilde{\Phi}$. 
This case will be proved in Proposition \ref{prop: case 1}. 

\noindent\textbf{Case 2}. 
$\|\langle\alpha_0, H\rangle\|\leq|\mu|^{-1}\leq  \|\langle\alpha_2, H\rangle\|\leq \|\langle\alpha_1, H\rangle\|$ and $| \langle\mu,\alpha_1\rangle|\lesssim |\mu|\approx |\langle\mu,\alpha_2\rangle|\approx |\langle\mu,\alpha_0\rangle|$.
This case will be proved in Proposition \ref{prop: case 2}. 

\noindent\textbf{Case 3}.  
$|\mu|^{-1}\leq \|\langle\alpha_0,H\rangle\|\leq \|\langle\alpha_2,H\rangle\|\leq \|\langle\alpha_1,H\rangle\|$ and $| \langle\mu,\alpha_1\rangle|\lesssim |\mu|\approx |\langle\mu,\alpha_2\rangle|\approx |\langle\mu,\alpha_0\rangle|$. 
This case will be proved in Proposition \ref{prop: case 3}. 
\end{proof}

We first prove \textbf{Case 1}. 

\begin{prop}[\textbf{Case 1}]\label{prop: case 1}
    Assume that $\|\langle\alpha, H\rangle\|\lesssim |\mu|^{-1}$ for all $\alpha\in\widetilde{\Phi}$. Then \eqref{eq: main bd} is true. 
\end{prop}
\begin{proof}
    We have $$\|\langle\alpha, H\rangle\|^{-1}\gtrsim |\mu|\gtrsim|\langle s\mu,\alpha\rangle|$$ for all $s\in W$ and $\alpha\in\widetilde{\Phi}$, and so the desired bound reduces to 
$$|\widetilde{\chi}(\mu, H)|\lesssim \sum_{s\in W}\prod_{\alpha\in\widetilde{\Phi}}|\langle s\mu,\alpha\rangle|=|W|\cdot \prod_{\alpha\in\widetilde{\Phi}}|\langle \mu,\alpha\rangle|.$$
This follows from the Weyl dimension formula:
$$|\widetilde{\chi}(\mu,H)|\leq |\widetilde{\chi}(\mu,0)|=\frac{\prod_{\alpha\in\widetilde{\Phi}}|\langle\mu,\alpha\rangle|}{\prod_{\alpha\in\widetilde{\Phi}}|\langle\rho,\alpha\rangle|}.$$
\end{proof}

Next we deal with \textbf{Case 2} and \textbf{Case 3}. A particular consequence of \eqref{eq: assumption on H} and \eqref{eq: assumption on mu} is that the $s=e$ summand on the right-hand side of \eqref{eq: main bd} dominates all the other summands, since the product is maximized when the two increasing sequences are paired in the same order.
Thus, under the assumptions \eqref{eq: assumption on H} and \eqref{eq: assumption on mu}, the main estimate \eqref{eq: main bd} is equivalent to 
\begin{align}\label{eq: main bd version 2}
|\widetilde{\chi}(\mu, H)|&\lesssim 
\prod_{\alpha\in\widetilde{\Phi}}\min\{|\langle\mu,\alpha\rangle|,\|\langle\alpha,H\rangle\|^{-1}\}.
\end{align}
We will prove the above estimate for   \textbf{Case 2} and \textbf{Case 3}  in the next two sections respectively.

\section{Proof of \textbf{Case 2}}\label{sec: case 2}

\begin{prop}[\textbf{Case 2}]\label{prop: case 2}
    Assume that $\|\langle\alpha_0, H\rangle\|\leq|\mu|^{-1}\leq  \|\langle\alpha_2, H\rangle\|\leq \|\langle\alpha_1, H\rangle\|$ and $| \langle\mu,\alpha_1\rangle|\lesssim |\mu|\approx |\langle\mu,\alpha_2\rangle|\approx |\langle\mu,\alpha_0\rangle|$. Then \eqref{eq: main bd version 2} is true. 
\end{prop}
\begin{proof}
As $\|\langle\alpha_0, H\rangle\|^{-1}\geq|\mu|\approx|\langle\mu,\alpha_0\rangle|$ and $\|\langle\alpha_2, H\rangle\|^{-1}\leq|\mu|\approx|\langle\mu,\alpha_2\rangle|$, the  bound \eqref{eq: main bd version 2}
becomes 
\begin{align}\label{eq: case 2 bound}
|\widetilde{\chi}(\mu, H)|
&\lesssim \frac{1}{\|\langle\alpha_2, H\rangle\|}|\mu|\min\{|\langle\mu,\alpha_1\rangle|,\|\langle\alpha_1, H\rangle\|^{-1}\}.
\end{align}

\textbf{Case 2a}. $|\langle\mu,\alpha_1\rangle|\geq \|\langle\alpha_1, H\rangle\|^{-1}$. The desired bound \eqref{eq: case 2 bound} becomes 
$$|\widetilde{\chi}(\mu, H)|\lesssim \frac{1}{\|\langle\alpha_2, H\rangle\|\cdot \|\langle\alpha_1, H\rangle\|}|\mu|.$$
Specializing the character descent formula \eqref{eq: cha} to $j=0$, we have 
$$|\widetilde{\chi}(\mu, H)|\lesssim \frac{1}{\|\langle\alpha_2, H\rangle\|\cdot\|\langle\alpha_1, H\rangle\|}\sum_{s_0\in W_0}|\widetilde{\chi}^0({(s_0\mu)^0}, H^0)|.$$
Applying \eqref{eq: rank one bd}, we obtain 
$$|\widetilde{\chi}^0({(s_0\mu)^0}, H^0)|\lesssim |(s_0\mu)^0|\lesssim |\mu|,$$
which concludes the proof of \eqref{eq: case 2 bound}. 

\textbf{Case 2b}. $|\langle\mu,\alpha_1\rangle|\leq \|\langle\alpha_1, H\rangle\|^{-1}$. The desired bound \eqref{eq: case 2 bound} becomes 
$$|\widetilde{\chi}(\mu, H)|\lesssim \frac{1}{\|\langle\alpha_2, H\rangle\|}|\mu|\cdot  |\langle\mu,\alpha_1\rangle|.$$
Specializing \eqref{eq: cha} to $j=0$ again, we obtain 
$$
   \widetilde{\chi}(\mu, H)=\frac{1}{4}\sum_{s_0\in W_0}I^0_{s_0},
   $$
   where 
   $$
   I^0_{s_0}=
   \frac{1}{\sin\left(\frac{\langle\alpha_1, H\rangle}{2}\right)\sin\left(\frac{\langle\alpha_2, H\rangle}{2}\right)}(\det s_0) 
e^{i\langle s_0\mu, H_{0}\rangle}\widetilde{\chi}^0({(s_0\mu)^0}, H^0).
$$
Recall that $W_0=\{e,s_{\alpha_1},s_{\alpha_1}s_{\alpha_0}\}$. 
We first bound $I^0_{s_{\alpha_1}s_{\alpha_0}}$, and then exploit cancellation between $I^0_{e}$ and $I^0_{s_{\alpha_1}}$ to bound their sum $I^0_e+I^0_{s_{\alpha_1}}$. 

Observe that 
\begin{align*}
(s_{\alpha_1}s_{\alpha_0}\mu)^0=\frac{\langle s_{\alpha_1}s_{\alpha_0}\mu, \alpha_0  \rangle}{|\alpha_0|^2}  \alpha_0
    =\frac{\langle \mu, s_{\alpha_0}s_{\alpha_1}\alpha_0  \rangle}{|\alpha_0|^2}  \alpha_0
    =\frac{\langle \mu, \alpha_1  \rangle}{|\alpha_0|^2}  \alpha_0.
\end{align*}
Applying \eqref{eq: rank one bd}, the above identity implies 
$$|\widetilde{\chi}^0({(s_{\alpha_1}s_{\alpha_0}\mu)^0}, H^0)|\lesssim |\langle \mu, \alpha_1  \rangle|.$$
This combined with the assumption that $\|\langle\alpha_1, H\rangle\|\geq |\mu|^{-1}$ yields the desired bound for $I^0_{s_{\alpha_1}s_{\alpha_0}}$:
$$|I^0_{s_{\alpha_1}s_{\alpha_0}}|\lesssim \frac{1}{\|\langle\alpha_2, H\rangle\|}|\mu|\cdot |\langle \mu, \alpha_1  \rangle|.$$

It suffices to get the same bound for $I^0_{e}+I^0_{s_{\alpha_1}}$. Write 
\begin{align}\label{eq: alpha0H}
    \langle\alpha_0, H\rangle=2\pi k_0+f_0(H)
\end{align}
where $k_0\in\mathbb{Z}$ is chosen so that 
\begin{align}\label{eq: f0 case 2}
    |f_0(H)|\approx \|\langle\alpha_0, H\rangle\|\leq |\mu|^{-1}.
\end{align}

As $\mu\in\Lambda$, $2\langle\mu,\alpha_1\rangle/|\alpha_1|^2$ is an integer. 
The argument below will depend on the parity of both $k_0$ and $2\langle\mu,\alpha_1\rangle/|\alpha_1|^2$. Let 
\begin{align}\label{eq: def delta case 2}
    \delta_0=\left\{\begin{array}{ll}
    1, & \text{ if both }k_0 \text{ and } 2\frac{\langle\mu,\alpha_1\rangle}{|\alpha_1|^2}\text{ are odd}, \\
    0, & \text{ otherwise}.
\end{array}\right.
\end{align}

We have 
\begin{align}
    I^0_{e}+I^0_{s_{\alpha_1}}=&\frac{1}{\sin\left(\frac{\langle\alpha_1, H\rangle}{2}\right)\sin\left(\frac{\langle\alpha_2, H\rangle}{2}\right)}\nonumber\\
    &\cdot 
    \left(e^{i\langle \mu, H_{0}\rangle}\widetilde{\chi}^0({\mu^0}, H^0)-e^{i\langle s_{\alpha_1}\mu, H_{0}\rangle}\widetilde{\chi}^0({(s_{\alpha_1}\mu)^0}, H^0)\right)\nonumber\\
    =&J^0_1+J^0_2,\label{eq: cancellation case 2}
    \end{align}
    where 
     $$J^0_1=\frac{1}{ \sin\left(\frac{\langle\alpha_1, H\rangle}{2}\right)\sin\left(\frac{\langle\alpha_2, H\rangle}{2}\right)}\left(e^{i\langle \mu, H_{0}\rangle}-(-1)^{\delta_0} e^{i\langle s_{\alpha_1}\mu, H_{0}\rangle}\right)\widetilde{\chi}^0({\mu^0}, H^0),$$
    $$J^0_2=\frac{1}{ \sin\left(\frac{\langle\alpha_1, H\rangle}{2}\right)\sin\left(\frac{\langle\alpha_2, H\rangle}{2}\right)}
    e^{i\langle s_{\alpha_1}\mu, H_{0}\rangle}\left((-1)^{\delta_0}\widetilde{\chi}^0({\mu^0}, H^0)-\widetilde{\chi}^0({(s_{\alpha_1}\mu)^0}, H^0)\right).$$
We treat the desired bounds for $J^0_1$ and $J^0_2$ in Propositions \ref{prop: J1} and \ref{prop: J2} respectively, and this finishes the proof of the main proposition.  
\end{proof}

    \begin{prop}\label{prop: J1}
    Assume that $\|\langle\alpha_0, H\rangle\|\leq|\mu|^{-1}\leq  \|\langle\alpha_2, H\rangle\|\leq \|\langle\alpha_1, H\rangle\|$ and $| \langle\mu,\alpha_1\rangle|\lesssim |\mu|\approx |\langle\mu,\alpha_2\rangle|\approx |\langle\mu,\alpha_0\rangle|$.  Then 
        $$|J^0_1|\lesssim \frac{1}{\|\langle\alpha_2, H\rangle\|}|\mu|\cdot |\langle \mu, \alpha_1  \rangle|.$$
    \end{prop}
   
\begin{proof}
Noting that by \eqref{eq: rank one bd} we have $|\widetilde{\chi}^0({\mu^0}, H^0)|\lesssim|\mu^0|\lesssim |\mu|$, it suffices to prove  
\begin{align}\label{eq: proving J1}
    \left|e^{i\langle \mu, H_{0}\rangle}-(-1)^{\delta_0} e^{i\langle s_{\alpha_1}\mu, H_{0}\rangle}\right|
    \lesssim |\langle\mu,\alpha_1\rangle |\cdot \|\langle\alpha_1,H\rangle\|.
\end{align}
Observe that 
\begin{align}\label{eq: simple difference}
    \mu-s_{\alpha_1}\mu=2\frac{\langle\mu,\alpha_1\rangle}{|\alpha_1|^2}\alpha_1.
\end{align}
Then 
\begin{align*}
  \left|e^{i\langle \mu, H_{0}\rangle}-(-1)^{\delta_0} e^{i\langle s_{\alpha_1}\mu, H_{0}\rangle}\right|
&=\left|e^{i\langle \mu, H_{0}\rangle-i\langle s_{\alpha_1}\mu, H_{0}\rangle+i{\delta_0}\pi }-1\right|\\
&=\left|e^{i2\frac{\langle\mu,\alpha_1\rangle}{|\alpha_1|^2}\langle \alpha_1, H_{0}\rangle+i{\delta_0}\pi }-1\right|\\
&=2\left|\sin\left(\frac{\langle\mu,\alpha_1\rangle}{|\alpha_1|^2}\langle \alpha_1, H_{0}\rangle+\frac{{\delta_0} \pi}{2}\right)\right|.
\end{align*}
Recalling \eqref{eq: alpha0H}, we have 
$$\langle\alpha_1,H^0\rangle=\frac{\langle\alpha_0, H\rangle}{|\alpha_0|^2}\langle\alpha_1,\alpha_0\rangle=-\frac{1}{2}\langle\alpha_0, H\rangle=-\pi k_0-\frac12 f_0(H).$$
Thus 
\begin{align*}
   -\frac{\langle\mu,\alpha_1\rangle}{|\alpha_1|^2}\langle \alpha_1, H^{0}\rangle+\frac{{\delta_0} \pi}{2}
    &=\frac{\pi}{2}\left({\delta_0}+ \frac{2\langle\mu,\alpha_1\rangle }{|\alpha_1|^2}k_0\right)+ \frac{\langle\mu,\alpha_1\rangle}{2|\alpha_1|^2}f_0(H).
\end{align*}
Recalling from the definition \eqref{eq: def delta case 2} of ${\delta_0}$ that 
${\delta_0}+ \frac{2\langle\mu,\alpha_1\rangle }{|\alpha_1|^2}k_0$ is always an even number, the above identity implies
\begin{align*}
    &\left|\sin\left(\frac{\langle\mu,\alpha_1\rangle}{|\alpha_1|^2}\langle \alpha_1, H_{0}\rangle+\frac{{\delta_0}\pi}{2}\right)\right|\\
    &=\left|\sin\left(\frac{\langle\mu,\alpha_1\rangle}{|\alpha_1|^2}\langle \alpha_1, H\rangle-  \frac{\langle\mu,\alpha_1\rangle}{|\alpha_1|^2}\langle \alpha_1, H^{0}\rangle +\frac{{\delta_0} \pi}{2}\right)\right|\\
    &=\left|\sin\left(\frac{\langle\mu,\alpha_1\rangle}{|\alpha_1|^2}\langle \alpha_1, H\rangle+ \frac{\langle\mu,\alpha_1\rangle}{2|\alpha_1|^2}f_0(H)\right)\right|.
\end{align*}
Noting that $\langle\mu,\alpha_1\rangle/|\alpha_1|^2\in \frac12\mathbb{Z}$, in the evaluation of the above quantity we may assume 
$\langle \alpha_1, H\rangle\in[-\pi,\pi]$ and thus $|\langle \alpha_1, H\rangle|\approx\|\langle \alpha_1, H\rangle\|$.  
Then using $|\sin x|\leq |x|$, we have 
\begin{align*}
    \left|\sin\left(\frac{\langle\mu,\alpha_1\rangle}{|\alpha_1|^2}\langle \alpha_1, H_{0}\rangle+\frac{\delta_0 \pi}{2}\right)\right|\lesssim 
|\langle\mu,\alpha_1\rangle|\cdot \|\langle \alpha_1, H\rangle\|+ |\langle\mu,\alpha_1\rangle|\cdot |f_0(H)|.
\end{align*}
By \eqref{eq: f0 case 2}, we have $|f_0(H)|\lesssim |\mu|^{-1}\leq \|\langle\alpha_1,H\rangle\|$, and thus 
\begin{align*}
    \left|\sin\left(\frac{\langle\mu,\alpha_1\rangle}{|\alpha_1|^2}\langle \alpha_1, H_{0}\rangle+\frac{\delta_0 \pi}{2}\right)\right|
    \lesssim |\langle\mu,\alpha_1\rangle |\cdot \|\langle\alpha_1,H\rangle\|.
\end{align*}
We have verified \eqref{eq: proving J1} and thus finished the proof.     
\end{proof}

    \begin{prop}\label{prop: J2}
    Assume that $\|\langle\alpha_0, H\rangle\|\leq|\mu|^{-1}\leq  \|\langle\alpha_2, H\rangle\|\leq \|\langle\alpha_1, H\rangle\|$ and $| \langle\mu,\alpha_1\rangle|\lesssim |\mu|\approx |\langle\mu,\alpha_2\rangle|\approx |\langle\mu,\alpha_0\rangle|$. Then 
        $$|J^0_2|\lesssim \frac{1}{\|\langle\alpha_2, H\rangle\|}|\mu|\cdot |\langle \mu, \alpha_1  \rangle|.$$
    \end{prop}
    \begin{proof}
        It suffices to prove that 
        \begin{align}\label{eq: proving J2}
            \left|(-1)^{\delta_0}\widetilde{\chi}^0({\mu^0}, H^0)-\widetilde{\chi}^0({(s_{\alpha_1}\mu)^0}, H^0)\right|
            \lesssim |\langle\mu,\alpha_1\rangle|.
        \end{align}
Using $|\sin x-\sin y|\leq 2|\sin\frac{x-y}{2}|$, we have 
\begin{align*}
 &\left|(-1)^{\delta_0}\widetilde{\chi}^0({\mu^0}, H^0)-\widetilde{\chi}^0({(s_{\alpha_1}\mu)^0}, H^0)\right|\\
 &=\frac1{\|\langle\alpha_0,H\rangle\|}\left|\sin(\langle\mu^0,H^0\rangle+{\delta_0}\pi)-\sin(\langle (s_{\alpha_1}\mu)^0,H^0\rangle)\right|\\
 &\leq \frac2{\|\langle\alpha_0,H\rangle\|}\left|\sin\left(\frac12(\langle\mu^0,H^0\rangle+{\delta_0}\pi-\langle (s_{\alpha_1}\mu)^0,H^0\rangle)\right)\right|.
\end{align*}
Now compute
$$\langle\mu^0,H^0\rangle=\frac{\langle\mu,\alpha_0\rangle}{|\alpha_0|^2}\langle\alpha_0,H\rangle,$$
$$\langle (s_{\alpha_1}\mu)^0,H^0\rangle =\frac{\langle s_{\alpha_1}\mu,\alpha_0\rangle}{|\alpha_0|^2}\langle\alpha_0,H\rangle,$$
which together imply, using \eqref{eq: simple difference} and \eqref{eq: alpha0H}, that  
\begin{align*}
    \langle\mu^0,H^0\rangle-\langle (s_{\alpha_1}\mu)^0,H^0\rangle &=
2\frac{\langle\mu,\alpha_1\rangle}{|\alpha_1|^2|\alpha_0|^2}\langle\alpha_1,\alpha_0\rangle\langle\alpha_0,H\rangle\\
&=-\frac{\langle\mu,\alpha_1\rangle}{|\alpha_1|^2}\langle\alpha_0,H\rangle\\
&=-\frac{\langle\mu,\alpha_1\rangle}{|\alpha_1|^2}2\pi k_0-\frac{\langle\mu,\alpha_1\rangle}{|\alpha_1|^2}f_0(H).
\end{align*}
Then 
\begin{align*}
   &\left|\sin\left(\frac12(\langle\mu^0,H^0\rangle+{\delta_0}\pi-\langle (s_{\alpha_1}\mu)^0,H^0\rangle)\right)\right|\\
   &=\left|\sin\left(\frac{\pi}{2}\left({\delta_0}-\frac{2\langle\mu,\alpha_1\rangle}{|\alpha_1|^2}k_0\right)-\frac{\langle\mu,\alpha_1\rangle}{2|\alpha_1|^2}f_0(H)\right)\right|\\
   &=\left|\sin\left(\frac{\langle\mu,\alpha_1\rangle}{2|\alpha_1|^2}f_0(H)\right)\right|,
\end{align*}
where in the last equality we used the fact that ${\delta_0}-\frac{2\langle\mu,\alpha_1\rangle}{|\alpha_1|^2}k_0$ is an even number.
Now we have  
\begin{align*}
 \left|(-1)^{\delta_0}\widetilde{\chi}^0({\mu^0}, H^0)-\widetilde{\chi}^0({(s_{\alpha_1}\mu)^0}, H^0)\right|
 &\leq \frac2{\|\langle\alpha_0,H\rangle\|}\left|\sin\left(\frac{\langle\mu,\alpha_1\rangle}{2|\alpha_1|^2}f_0(H)\right)\right|\\
 &\leq\frac2{\|\langle\alpha_0,H\rangle\|} \frac{|\langle\mu,\alpha_1\rangle|}{2|\alpha_1|^2}|f_0(H)|\\
 &\lesssim  |\langle\mu,\alpha_1\rangle|.
\end{align*}
Here in the last inequality we used \eqref{eq: f0 case 2}. 
This concludes the proof of \eqref{eq: proving J2} and thus the proposition. 
    \end{proof}

\section{Proof of \textbf{Case 3}}\label{sec: case 3}
In the proof of \textbf{Case 3}, the argument will follow a structure similar to that of \textbf{Case 2}. The main difference, however, will lie in the choice of $j\in\{0,1,2\}$ indexing the parabolic subgroup of the Weyl group when applying Lemma \ref{lem: cha}. In the proof of \textbf{Case 2} in the previous section, we chose $j=0$, guided by the fact that $\|\langle\alpha_0,H\rangle\|$ is small: it is then natural to descend the character toward the alcove walls corresponding to $\|\langle\alpha_0,H\rangle\|=0$. The same choice $j=0$ will not work for \textbf{Case 3}. Instead, in the proof of \textbf{Case 3}, we will choose $j=1$ when applying Lemma \ref{lem: cha}. The main idea, in retrospect, is that $|\langle\mu,\alpha_1\rangle|$ is small and that the choice $j=1$ corresponds to this ``spectrally small'' root $\alpha_1$. We expect these ideas to play a role in extending the main result to more general compact Lie groups.

\begin{prop}[\textbf{Case 3}]\label{prop: case 3}
    Assume that $|\mu|^{-1}\leq \|\langle\alpha_0,H\rangle\|\leq \|\langle\alpha_2,H\rangle\|\leq \|\langle\alpha_1,H\rangle\|$ and $| \langle\mu,\alpha_1\rangle|\lesssim |\mu|\approx |\langle\mu,\alpha_2\rangle|\approx |\langle\mu,\alpha_0\rangle|$. Then \eqref{eq: main bd version 2} is true. 
\end{prop}
\begin{proof}
As $\|\langle\alpha_0, H\rangle\|^{-1}\leq |\mu|\approx|\langle\mu,\alpha_0\rangle|$ and $\|\langle\alpha_2, H\rangle\|^{-1}\leq|\mu|\approx|\langle\mu,\alpha_2\rangle|$, the bound \eqref{eq: main bd version 2} becomes 
\begin{align}\label{eq: case 3 bd}
    |\widetilde{\chi}(\mu, H)|\lesssim \frac{1}{\|\langle\alpha_2,H\rangle\|\cdot \|\langle\alpha_0,H\rangle\|}\min\{|\langle\mu,\alpha_1\rangle|,\|\langle\alpha_1,H\rangle\|^{-1}\}.
\end{align}

\textbf{Case 3a}. $|\langle\mu,\alpha_1\rangle|\geq\|\langle\alpha_1,H\rangle\|^{-1}$. The desired bound \eqref{eq: case 3 bd} becomes 
$$|\widetilde{\chi}(\mu, H)|\lesssim \frac{1}{\|\langle\alpha_2,H\rangle\|\cdot\|\langle\alpha_0,H\rangle\|\cdot\|\langle\alpha_1,H\rangle\|}.$$
This is a direct consequence of the Weyl character formula \eqref{eq: Weyl}. 

\textbf{Case 3b}. $|\langle\mu,\alpha_1\rangle|\leq\|\langle\alpha_1,H\rangle\|^{-1}$. 
 The desired bound \eqref{eq: case 3 bd} becomes 
$$|\widetilde{\chi}(\mu, H)|\lesssim \frac{1}{\|\langle\alpha_2,H\rangle\|\cdot\|\langle\alpha_0,H\rangle\|}|\langle\mu,\alpha_1\rangle|.$$
   Specializing \eqref{eq: cha} to $j=1$, we obtain
   $$
   \widetilde{\chi}(\mu, H)=\frac{1}{4}\sum_{s_1\in W_1}I^1_{s_1},
   $$
   where 
   $$
   I^1_{s_1}=
   \frac{1}{\sin\left(\frac{\langle\alpha_2, H\rangle}{2}\right)\sin\left(\frac{\langle\alpha_0, H\rangle}{2}\right)}(\det s_1) 
e^{i\langle s_1\mu, H_{1}\rangle}\widetilde{\chi}^1({(s_1\mu)^1}, H^1).
$$
Recall that $W_1=\{e,s_{\alpha_2},s_{\alpha_2}s_{\alpha_1}\}$. 
We first bound $I^1_{e}$, and then exploit cancellation between $I^1_{s_{\alpha_2}}$ and $I^1_{s_{\alpha_2}s_{\alpha_1}}$ to bound their sum $I^1_{s_{\alpha_2}}+I^1_{s_{\alpha_2}s_{\alpha_1}}$.

Applying \eqref{eq: rank one bd}, we obtain 
$$|\widetilde{\chi}^1(\mu^1,H^1)|\lesssim |\mu^1|\approx|\langle \mu,\alpha_1\rangle|.$$
This implies the desired bound for $I^1_{e}$:
$$|I^1_{e}|\lesssim \frac{1}{\|\langle\alpha_2,H\rangle\|\cdot \|\langle\alpha_0,H\rangle\|}|\langle\mu,\alpha_1\rangle|.$$

Next we bound $I^1_{s_{\alpha_2}}+I^1_{s_{\alpha_2}s_{\alpha_1}}$. 
Write 
\begin{align}\label{eq: alpha1H}
    \langle\alpha_1, H\rangle=2\pi k_1+f_1(H)
\end{align}
where $k_1\in\mathbb{Z}$ is chosen so that 
\begin{align}\label{eq: f1}
    |f_1(H)|\approx \|\langle\alpha_1, H\rangle\|.
\end{align}
Similar to the treatment of \textbf{Case 2} in Proposition \ref{prop: case 2}, the argument below will depend on the parity of both $k_1$ and $2\langle\mu,\alpha_1\rangle/|\alpha_1|^2$. Let 
\begin{align}\label{eq: def delta 1}
\delta_1=\left\{\begin{array}{ll}
    1, & \text{ if both }k_1 \text{ and } 2\frac{\langle\mu,\alpha_1\rangle}{|\alpha_1|^2}\text{ are odd}, \\
    0, & \text{ otherwise}.
\end{array}\right.
\end{align}

We have 
\begin{align}
   &I^1_{s_{\alpha_2}}+I^1_{s_{\alpha_2}s_{\alpha_1}}\nonumber\\
   &=\frac{1}{ \sin\left(\frac{\langle\alpha_2, H\rangle}{2}\right)\sin\left(\frac{\langle\alpha_0, H\rangle}{2}\right)}\nonumber\\
    &\cdot 
    \left(e^{i\langle s_{\alpha_2}\mu, H_{1}\rangle}\widetilde{\chi}^1({(s_{\alpha_2}\mu)^1}, H^1)-e^{i\langle s_{\alpha_2}s_{\alpha_1}\mu, H_{1}\rangle}\widetilde{\chi}^1({(s_{\alpha_2}s_{\alpha_1}\mu)^1}, H^1)\right)\nonumber\\
    &=J^1_1+J^1_2,\label{eq: cancellation case 3}
    \end{align}
    where 
     $$J^1_1=\frac{\widetilde{\chi}^1({(s_{\alpha_2}\mu)^1}, H^1)}{ \sin\left(\frac{\langle\alpha_2, H\rangle}{2}\right)\sin\left(\frac{\langle\alpha_0, H\rangle}{2}\right)}\left(e^{i\langle s_{\alpha_2}\mu, H_{1}\rangle}-(-1)^{\delta_1} e^{i\langle s_{\alpha_2}s_{\alpha_1}\mu, H_{1}\rangle}\right),$$
    $$J^1_2=\frac{e^{i\langle s_{\alpha_2}s_{\alpha_1}\mu, H_{1}\rangle}}{ \sin\left(\frac{\langle\alpha_2, H\rangle}{2}\right)\sin\left(\frac{\langle\alpha_0, H\rangle}{2}\right)}
    \left((-1)^{\delta_1}\widetilde{\chi}^1({(s_{\alpha_2}\mu)^1}, H^1)-\widetilde{\chi}^1({(s_{\alpha_2}s_{\alpha_1}\mu)^1}, H^1)\right).$$
We treat the desired bound for $J^1_1$ and $J^1_2$ in Propositions \ref{prop: tilde J1} and \ref{prop: tilde J2} respectively, and this finishes the proof of the main proposition.  
\end{proof}

\begin{prop}\label{prop: tilde J1}
Assume that $|\mu|^{-1}\leq \|\langle\alpha_0,H\rangle\|\leq \|\langle\alpha_2,H\rangle\|\leq \|\langle\alpha_1,H\rangle\|$ and $| \langle\mu,\alpha_1\rangle|\lesssim |\mu|\approx |\langle\mu,\alpha_2\rangle|\approx |\langle\mu,\alpha_0\rangle|$. Then
    $$|J^1_1|\lesssim \frac{1}{\|\langle\alpha_2,H\rangle\|\cdot \|\langle\alpha_0,H\rangle\|}|\langle\mu,\alpha_1\rangle|.$$
\end{prop}
\begin{proof}
Recalling \eqref{eq: rank one bd}, as $|\widetilde{\chi}^1({(s_{\alpha_2}\mu)^1}, H^1)|\lesssim \|\langle\alpha_1,H^1\rangle\|^{-1}=\|\langle\alpha_1,H\rangle\|^{-1}$, it suffices to prove that 
\begin{align}\label{eq: proving tilde J1}
    \left|e^{i\langle s_{\alpha_2}\mu, H_{1}\rangle}-(-1)^{\delta_1} e^{i\langle s_{\alpha_2}s_{\alpha_1}\mu, H_{1}\rangle}\right|
    \lesssim   |\langle\mu,\alpha_1\rangle |\cdot\|\langle\alpha_1,H\rangle\|.
\end{align}
Observe that 
\begin{align}\label{eq: simple difference case 3}
    s_{\alpha_2}\mu-s_{\alpha_2}s_{\alpha_1}\mu=2\frac{\langle\mu,\alpha_1\rangle}{|\alpha_1|^2}s_{\alpha_2}\alpha_1
    =-2\frac{\langle\mu,\alpha_1\rangle}{|\alpha_1|^2}\alpha_0.
\end{align}
Then 
\begin{align*}
  \left|e^{i\langle s_{\alpha_2}\mu, H_{1}\rangle}-(-1)^{\delta_1} e^{i\langle s_{\alpha_2}s_{\alpha_1}\mu, H_{1}\rangle}\right|
&=\left|e^{i\langle s_{\alpha_2}\mu, H_{1}\rangle-i\langle s_{\alpha_2}s_{\alpha_1}\mu, H_{1}\rangle+i{\delta_1}\pi }-1\right|\\
&=\left|e^{-i2\frac{\langle\mu,\alpha_1\rangle}{|\alpha_1|^2}\langle \alpha_0, H_{1}\rangle+i{\delta_1}\pi }-1\right|\\
&=2\left|\sin\left(-\frac{\langle\mu,\alpha_1\rangle}{|\alpha_1|^2}\langle \alpha_0, H_{1}\rangle+\frac{{\delta_1} \pi}{2}\right)\right|.
\end{align*}
From \eqref{eq: alpha1H} we have 
$$\langle\alpha_0,H^1\rangle=\frac{\langle\alpha_1, H\rangle}{|\alpha_1|^2}\langle\alpha_0,\alpha_1\rangle=-\frac{1}{2}\langle\alpha_1, H\rangle=-\pi k_1-\frac12 f_1(H).$$
Thus 
\begin{align*}
   \frac{\langle\mu,\alpha_1\rangle}{|\alpha_1|^2}\langle \alpha_0, H^{1}\rangle+\frac{{\delta_1} \pi}{2}
    &=\frac{\pi}{2}\left({\delta_1}- \frac{2\langle\mu,\alpha_1\rangle }{|\alpha_1|^2}k_1\right)- \frac{\langle\mu,\alpha_1\rangle}{2|\alpha_1|^2}f_1(H).
\end{align*}
Recalling from the definition \eqref{eq: def delta 1} of ${\delta_1}$ that ${\delta_1}- \frac{2\langle\mu,\alpha_1\rangle }{|\alpha_1|^2}k_1$ is always an even number, the above identity implies that 
\begin{align*}
    &\left|\sin\left(-\frac{\langle\mu,\alpha_1\rangle}{|\alpha_1|^2}\langle \alpha_0, H_{1}\rangle+\frac{{\delta_1}\pi}{2}\right)\right|\\
    &=\left|\sin\left(-\frac{\langle\mu,\alpha_1\rangle}{|\alpha_1|^2}\langle \alpha_0, H\rangle+  \frac{\langle\mu,\alpha_1\rangle}{|\alpha_1|^2}\langle \alpha_0, H^{1}\rangle +\frac{{\delta_1} \pi}{2}\right)\right|\\
    &=\left|\sin\left(\frac{\langle\mu,\alpha_1\rangle}{|\alpha_1|^2}\langle \alpha_0, H\rangle+ \frac{\langle\mu,\alpha_1\rangle}{2|\alpha_1|^2}f_1(H)\right)\right|.
\end{align*}
Noting that $\langle\mu,\alpha_1\rangle/|\alpha_1|^2\in \frac12\mathbb{Z}$, in the evaluation of the above quantity we may assume that 
$\langle \alpha_0, H\rangle\in[-\pi,\pi]$ and thus $|\langle \alpha_0, H\rangle|\approx \|\langle \alpha_0, H\rangle\|$. 
Then using $|\sin x|\leq |x|$, we have 
\begin{align*}
    \left|\sin\left(-\frac{\langle\mu,\alpha_1\rangle}{|\alpha_1|^2}\langle \alpha_0, H_{1}\rangle+\frac{{\delta_1} \pi}{2}\right)\right|
    \lesssim 
|\langle\mu,\alpha_1\rangle|\cdot \|\langle\alpha_0, H\rangle\|+ |\langle\mu,\alpha_1\rangle|\cdot |f_1(H)|.
\end{align*}
Using \eqref{eq: f1} and the assumption that 
$\|\langle\alpha_0,H\rangle\|\leq\|\langle\alpha_1,H\rangle\|$,
we get 
\begin{align*}
    \left|\sin\left(-\frac{\langle\mu,\alpha_1\rangle}{|\alpha_1|^2}\langle \alpha_0, H_{1}\rangle+\frac{{\delta_1} \pi}{2}\right)\right|
    \lesssim |\langle\mu,\alpha_1\rangle |\cdot \|\langle\alpha_1,H\rangle\|.
\end{align*}
We have verified \eqref{eq: proving tilde J1} and thus finished the proof.     
\end{proof}

\begin{prop}\label{prop: tilde J2}
Assume that $|\mu|^{-1}\leq \|\langle\alpha_0,H\rangle\|\leq \|\langle\alpha_2,H\rangle\|\leq \|\langle\alpha_1,H\rangle\|$ and $| \langle\mu,\alpha_1\rangle|\lesssim |\mu|\approx |\langle\mu,\alpha_2\rangle|\approx |\langle\mu,\alpha_0\rangle|$. Then
    $$|J^1_2|\lesssim \frac{1}{\|\langle\alpha_2,H\rangle\|\cdot \|\langle\alpha_0,H\rangle\|}|\langle\mu,\alpha_1\rangle|.$$
\end{prop}
\begin{proof}
It suffices to prove that 
\begin{align}\label{eq: proving J2 tilde}
    \left|(-1)^{\delta_1}\widetilde{\chi}^1({(s_{\alpha_2}\mu)^1}, H^1)-\widetilde{\chi}^1({(s_{\alpha_2}s_{\alpha_1}\mu)^1}, H^1)\right|
    \lesssim |\langle\mu,\alpha_1\rangle|.
\end{align}
Using $|\sin x-\sin y|\leq 2|\sin\frac{x-y}2|$, we have 
\begin{align}\label{eq: initial estimate}
 &\left|(-1)^{\delta_1}\widetilde{\chi}^1({(s_{\alpha_2}\mu)^1}, H^1)-\widetilde{\chi}^1({(s_{\alpha_2}s_{\alpha_1}\mu)^1}, H^1)\right|\nonumber\\
 &=\frac1{\|\langle\alpha_1,H\rangle\|}\left|\sin(\langle(s_{\alpha_2}\mu)^1,H^1\rangle+{\delta_1}\pi)-\sin(\langle (s_{\alpha_2}s_{\alpha_1}\mu)^1,H^1\rangle)\right|\nonumber\\
 &\leq \frac2{\|\langle\alpha_1,H\rangle\|}\left|\sin\left(\frac12(\langle(s_{\alpha_2}\mu)^1,H^1\rangle+{\delta_1}\pi-\langle (s_{\alpha_2}s_{\alpha_1}\mu)^1,H^1\rangle)\right)\right|.
\end{align}
Now compute
$$\langle(s_{\alpha_2}\mu)^1,H^1\rangle=\frac{\langle s_{\alpha_2}\mu,\alpha_1\rangle}{|\alpha_1|^2}\langle\alpha_1,H\rangle,$$
$$\langle(s_{\alpha_2}s_{\alpha_1}\mu)^1,H^1\rangle=\frac{\langle s_{\alpha_2}s_{\alpha_1}\mu,\alpha_1\rangle}{|\alpha_1|^2}\langle\alpha_1,H\rangle,$$
which together imply, using \eqref{eq: simple difference case 3} and \eqref{eq: alpha1H}, that  
\begin{align*}
\langle(s_{\alpha_2}\mu)^1,H^1\rangle-\langle(s_{\alpha_2}s_{\alpha_1}\mu)^1,H^1\rangle
    &=-
2\frac{\langle\mu,\alpha_1\rangle}{|\alpha_1|^2|\alpha_1|^2}\langle\alpha_0,\alpha_1\rangle\langle\alpha_1,H\rangle\\
&=\frac{\langle\mu,\alpha_1\rangle}{|\alpha_1|^2}\langle\alpha_1,H\rangle\\
&=\frac{\langle\mu,\alpha_1\rangle}{|\alpha_1|^2}2\pi k_1+\frac{\langle\mu,\alpha_1\rangle}{|\alpha_1|^2}f_1(H).
\end{align*}
Using the fact that  $\frac{2\langle\mu,\alpha_1\rangle}{|\alpha_1|^2}k_1+{\delta_1}$ is always an even number, we then have 
\begin{align*}
   &\left|\sin\left(\frac12(\langle(s_{\alpha_2}\mu)^1,H^1\rangle+{\delta_1}\pi-\langle (s_{\alpha_2}s_{\alpha_1}\mu)^1,H^1\rangle)\right)\right|\\
   =&\left|\sin\left(\frac{\pi}{2}\left(\frac{2\langle\mu,\alpha_1\rangle}{|\alpha_1|^2}k_1+{\delta_1}\right)+\frac{\langle\mu,\alpha_1\rangle}{2|\alpha_1|^2}f_1(H)\right)\right|\\
   =&\left|\sin\left(\frac{\langle\mu,\alpha_1\rangle}{2|\alpha_1|^2}f_1(H)\right)\right|\\
   \lesssim& |\langle\mu,\alpha_1\rangle|\cdot |f_1(H)|\\
   \lesssim &|\langle\mu,\alpha_1\rangle|\cdot  \|\langle\alpha_1,H\rangle\|.
\end{align*}
Here in the last inequality we used \eqref{eq: f1}.
Applying the above estimate to \eqref{eq: initial estimate}, we obtain the desired estimate \eqref{eq: proving J2 tilde} and conclude the proof.  
    \end{proof}

\section{\texorpdfstring{$L^p$}{Lp} bounds for characters}\label{sec: Lp}
Let $A=\{H\in\mathfrak{t}: \langle\alpha_1,H\rangle\geq 0,\ \langle\alpha_2,H\rangle\geq 0,\ -\langle\alpha_0,H\rangle\leq 2\pi\}$ be a Weyl alcove. 
Applying the Weyl integration formula \cite{Wey39} along with \eqref{eq: second}, we have 
\begin{align}\label{eq: chi Lp first}
    \|\chi_\mu\|_{L^p(G)}^p
   &=4\int_A |\chi(\mu,H)|^p \prod_{\alpha\in\widetilde{\Phi}}\|\langle\alpha,H\rangle\|^2\,dH \nonumber\\
   &\lesssim d_\mu^p\int_A \sum_{s\in W}\prod_{\alpha\in\widetilde{\Phi}}
   \frac{\|\langle\alpha,H\rangle\|^2}{(1+|\langle s(\mu+\rho),\alpha\rangle|\cdot\|\langle\alpha, H\rangle\|)^{p}}
   \,dH.
\end{align}
Observe that the above integrand is invariant under the extended affine Weyl group: it is invariant not only under the Weyl group action \(H \mapsto sH\) (\(s \in W\)), but also under translations by the coweight lattice
\[
\{ H \in \mathfrak{t} : \langle \alpha, H \rangle \in 2\pi \mathbb{Z} \text{ for all }\alpha\in\widetilde{\Phi}\}.
\]
Consequently, without loss of generality, to obtain an upper bound for the above integral, we may replace the alcove $A$ by the neighborhood
\[
A_0 = \left\{ H \in \mathfrak{t} : \langle \alpha_1, H \rangle \geq 0,\ \langle \alpha_2, H \rangle \geq 0,\ -\langle \alpha_0, H \rangle \leq \tfrac{4\pi}{3} \right\}
\]
of $H=0$ in $A$, since the translates of $A_0$ under the extended affine Weyl group already cover $A$ (Figure \ref{fig: alcoves}). 

\begin{figure}[htbp]
\centering
\begin{tikzpicture}[scale=1.2, line cap=round, line join=round]

\def\s{2}
\def\h{1.7320508} 

\coordinate (O)  at (0,0);

\coordinate (L)  at (-\s,0);
\coordinate (LT) at (-1,\h);

\coordinate (R)  at (\s,0);
\coordinate (RT) at (1,\h);

\coordinate (B1) at (-1,-\h);
\coordinate (B2) at (1,-\h);

\draw[thick] (L)--(LT)--(O)--cycle;
\draw[thick] (O)--(RT)--(R)--cycle;
\draw[thick] (O)--(B1)--(B2)--cycle;

\fill[gray!60] (0,0) -- ({4/3},0) -- ({2/3},{2*\h/3})  -- cycle;

\draw ({1/3},{\h/3}) -- ({5/3},{\h/3});
\draw ({2/3},0) -- ({4/3},{2*\h/3});
\draw ({2/3},{2*\h/3}) -- ({4/3},0);

\draw ({-4/3},{2*\h/3}) -- ({-2/3},0);

\draw ({-1/3},{-\h/3}) -- ({1/3},{-\h});           

\end{tikzpicture}

\caption{Translates of $A_0$ under the extended affine Weyl group; $A_0$ is the shaded region.}
\label{fig: alcoves}
\end{figure}
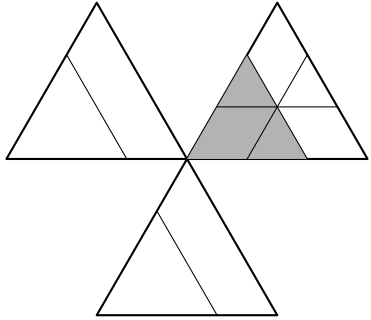

For $s \in W$, let
\[
\widetilde{a} = |\langle s(\mu+\rho), \alpha_1 \rangle|,\quad
\widetilde{b} = |\langle s(\mu+\rho), \alpha_2 \rangle|,\quad
\widetilde{c} = |\langle s(\mu+\rho), \alpha_0 \rangle|.
\]
For $H\in\mathfrak{t}$, we introduce the coordinates
$$t_1=\langle\alpha_1,H\rangle, \quad t_2=\langle\alpha_2,H\rangle.$$
Then $A_0=\left\{t_1,t_2\geq 0, \ t_1+t_2\leq \frac{4\pi}{3}\right\}$, and for $H\in A_0$, 
$$t_1\approx \|\langle\alpha_1,H\rangle\|,\quad t_2\approx \|\langle\alpha_2,H\rangle\|,\quad t_1+t_2\approx \|\langle\alpha_0,H\rangle\|.$$
We aim to bound 
\begin{align}
\mathcal{I}(s):=&\int_{A_0} \prod_{\alpha\in\widetilde{\Phi}}
   \frac{\|\langle\alpha,H\rangle\|^2}{(1+|\langle s(\mu+\rho),\alpha\rangle|\cdot\|\langle\alpha, H\rangle\|)^{p}}
   \,dH \label{eq: ABC integral} \\
   \lesssim& \int_{A_0}\frac{t_1^2t_2^2(t_1+t_2)^2}{(1+\widetilde{a}t_1)^p(1+\widetilde{b}t_2)^p(1+\widetilde{c}(t_1+t_2))^p}\,dt_1\,dt_2. \label{eq: ABC integral 2}
\end{align}

\begin{prop}\label{prop: I}
Let $(a,b,c)$ be the relabeling of the set 
$\{ |\langle \mu+\rho, \alpha \rangle| : \alpha \in \widetilde{\Phi} \}$ such that $a\ge b\ge c$. Then for all $s\in W$, we have 
    \[
\mathcal I(s)\lesssim
\begin{cases}
a^{-p}b^{-p}c^{-p}, & 0<p<\frac83,\\[1mm]
a^{-p}b^{-p}c^{-p}\log(2+c), & p=\frac83,\\[1mm]
a^{-p}b^{-p}c^{2p-8}, & \frac83<p<3,\\[1mm]
a^{-3}b^{-3}c^{-2}\log\left(2+\frac ac\right), & p=3, \\[1mm]
a^{-3}b^{-p}c^{p-5}, & 3<p<5,\\[1mm]
a^{-3}b^{-5}\log\left(2+\frac bc\right), & p=5,\\[1mm]
a^{-3}b^{-5}, & p>5.
\end{cases}
\]
\end{prop}
\begin{proof}
As $s$ varies over $W$, the triple $(\widetilde{a},\widetilde{b},\widetilde{c})$ runs over all permutations of the fixed set
$\{ |\langle \mu+\rho, \alpha \rangle| : \alpha \in \widetilde{\Phi} \}$. Observe that the integral in \eqref{eq: ABC integral 2} is maximized when
\[
(\widetilde a,\widetilde b,\widetilde c)=(a,b,c)
\quad\text{or equivalently }
(b,a,c).
\]
Thus it suffices to get the desired estimate for the integral in \eqref{eq: ABC integral 2} when $(\widetilde a,\widetilde b,\widetilde c)=(a,b,c)$. 

Splitting \(A_0\) into \(t_1\ge t_2\) and \(t_1\le t_2\), we may further assume 
\[t_1\le t_2,\]
since this region yields the larger contribution under the assumption $a\geq b$. Then $t_1+t_2\approx t_2$. 

Let $M=\frac{4\pi}{3}$. We have 
\begin{align*}
    \mathcal{I}=\mathcal{I}(s)&\lesssim \int_{0\leq t_1\leq t_2\leq M} 
    \frac{t_1^2t_2^4}{(1+at_1)^p(1+bt_2)^p(1+ct_2)^p}
    \,dt_1\,dt_2\\
    &=\int_0^M \frac{t_2^4}{(1+bt_2)^p(1+ct_2)^p}
\left(\int_0^{t_2}\frac{t_1^2}{(1+at_1)^p}\,dt_1\right)\,dt_2.
\end{align*}

\underline{The case of \(0<p<3\).} We have
\[
\int_0^{t_2}\frac{t_1^2}{(1+at_1)^p}\,dt_1
\lesssim
\begin{cases}
t_2^3, & 0\le t_2\le a^{-1},\\[1mm]
a^{-p}t_2^{3-p}, & t_2\ge a^{-1}.
\end{cases}
\]
Recall that $a \ge b \ge c > 0$. We split the $t_2$-integral at 
$t_2 = a^{-1}, b^{-1}, c^{-1}$ to obtain 
\begin{align*}
\mathcal I
\lesssim&
\int_0^{a^{-1}} t_2^7\,dt_2
+a^{-p}\int_{a^{-1}}^{b^{-1}} t_2^{7-p}\,dt_2\\
& +a^{-p}b^{-p}\int_{b^{-1}}^{c^{-1}} t_2^{7-2p}\,dt_2
+a^{-p}b^{-p}c^{-p}\int_{c^{-1}}^M t_2^{7-3p}\,dt_2.
\end{align*}
For the first three terms, we have 
\begin{align*}
    & \int_0^{a^{-1}} t_2^7\,dt_2
+a^{-p}\int_{a^{-1}}^{b^{-1}} t_2^{7-p}\,dt_2
+a^{-p}b^{-p}\int_{b^{-1}}^{c^{-1}} t_2^{7-2p}\,dt_2\\
\lesssim & a^{-8}+a^{-p}b^{p-8}+a^{-p}b^{-p}c^{2p-8}.
\end{align*}
For the last term, we have 
$$a^{-p}b^{-p}c^{-p}\int_{c^{-1}}^M t_2^{7-3p}\,dt_2\lesssim \begin{cases}
a^{-p}b^{-p}c^{-p}, & 0<p<\frac83,\\[1mm]
a^{-p}b^{-p}c^{-p}\log(2+c), & p=\frac83,\\[1mm]
a^{-p}b^{-p}c^{2p-8}, & \frac83<p<3.
\end{cases}$$
Using $a\geq b\geq c$,  we obtain
\[
\mathcal I\lesssim
\begin{cases}
a^{-p}b^{-p}c^{-p}, & 0<p<\frac83,\\[1mm]
a^{-p}b^{-p}c^{-p}\log(2+c), & p=\frac83,\\[1mm]
a^{-p}b^{-p}c^{2p-8}, & \frac83<p<3.
\end{cases}
\]

\underline{The case of $p=3$.} We have
\begin{align*}
\int_0^{t_2}\frac{t_1^2}{(1+at_1)^3}\,dt_1
\lesssim
\begin{cases}
t_2^3, & 0\le t_2\le a^{-1},\\[4pt]
a^{-3}\log(2+at_2), & t_2\ge a^{-1}.
\end{cases}
\end{align*}
Then we have 
\begin{align*}
    \mathcal{I}
    \lesssim &\int_0^{a^{-1}}t_2^7\, dt_2
    +a^{-3}\int_{a^{-1}}^{b^{-1}}t_2^4\log(2+at_2)\, dt_2\\
    &+a^{-3}b^{-3}\int_{b^{-1}}^{c^{-1}}t_2\log(2+at_2)\, dt_2
    +a^{-3}b^{-3}c^{-3}\int_{c^{-1}}^{M}t_2^{-2}\log(2+at_2)\, dt_2.
\end{align*}
The last integral above may be estimated by 
\begin{align*}
 \int_{c^{-1}}^{M}t_2^{-2}\log(2+at_2)\,dt_2&
 \lesssim \int_{c^{-1}}^{M}t_2^{-2}\left[\log\left(2+\frac ac\right)+\log(ct_2)\right] \,dt_2\\
 &\lesssim c\log\left(2+\frac ac\right),
\end{align*}
so that 
\begin{align*}
    \mathcal{I}
    \lesssim &  a^{-8}+a^{-3}b^{-5}\log\left(2+\frac ab\right)
    +a^{-3}b^{-3}c^{-2}\log\left(2+\frac ac\right).
\end{align*}
Since $a\ge b\ge c$, the last term dominates, and hence
\begin{align*}
\mathcal I
\lesssim a^{-3}b^{-3}c^{-2}\log\!\left(2+\frac ac\right).
\end{align*}

\underline{The case of \(p>3\).} We have 
\[
\int_0^{t_2}\frac{t_1^2}{(1+at_1)^p}\,dt_1
\lesssim
\begin{cases}
t_2^3, & 0\le t_2\le a^{-1},\\[1mm]
a^{-3}, & t_2\ge a^{-1}.
\end{cases}
\]
Hence
\begin{align*}
\mathcal I
\lesssim & 
\int_0^{a^{-1}} t_2^7\,dt_2
+a^{-3}\int_{a^{-1}}^{b^{-1}} t_2^4\,dt_2
 \\
&+a^{-3}b^{-p}\int_{b^{-1}}^{c^{-1}} t_2^{4-p}\,dt_2+a^{-3}b^{-p}c^{-p}\int_{c^{-1}}^M t_2^{4-2p}\,dt_2.
\end{align*}
For the third term, we have 
\[
a^{-3}b^{-p}\int_{b^{-1}}^{c^{-1}} t_2^{4-p}\,dt_2
\lesssim
\begin{cases}
a^{-3}b^{-p}c^{p-5}, & 3<p<5,\\[1mm]
a^{-3}b^{-5}\log\left(2+\frac bc\right), & p=5,\\[1mm]
a^{-3}b^{-5}, & p>5.
\end{cases}
\]
For the remaining terms, we have 
\begin{align*}
    & \int_0^{a^{-1}} t_2^7\,dt_2
+a^{-3}\int_{a^{-1}}^{b^{-1}} t_2^4\,dt_2
 +a^{-3}b^{-p}c^{-p}\int_{c^{-1}}^M t_2^{4-2p}\,dt_2 \\ 
 \lesssim & a^{-8}+a^{-3}b^{-5}+a^{-3}b^{-p}c^{p-5}. 
\end{align*}
Using $a\geq b\geq c$, we obtain
\[
\mathcal I\lesssim
\begin{cases}
a^{-3}b^{-p}c^{p-5}, & 3<p<5,\\[1mm]
a^{-3}b^{-5}\log\left(2+\frac bc\right), & p=5,\\[1mm]
a^{-3}b^{-5}, & p>5.
\end{cases}
\]
The proof is completed. 
\end{proof}

We can now conclude:

\begin{proof}[Proof of Theorem \ref{thm: Lp}]
For $\mu\in\Lambda^+$, again let $(a,b,c)$ be the relabeling of 
$\{|\langle \mu+\rho,\alpha\rangle| : \alpha \in \widetilde{\Phi}\}$ such that $a\geq b\geq c$. 
By the Weyl dimension formula, for $\mu\in\Lambda^+$, we have 
$$d_\mu=\frac{\prod_{\alpha\in\widetilde{\Phi}}|\langle\mu+\rho,\alpha\rangle|}{\prod_{\alpha\in\widetilde{\Phi}}|\langle\rho,\alpha\rangle|}\approx abc.$$ 
Recalling \eqref{eq: chi Lp first} and \eqref{eq: ABC integral}, we have 
\begin{align}\label{eq: last bound}
    \|\chi_\mu\|_{L^p(G)}
\lesssim abc\cdot \sum_{s\in W}\mathcal{I}(s)^{1/p}.
\end{align}
Now let $\overline{\mu}$ and $\underline{\mu}$ denote the maximum and minimum of 
$\{|\langle \mu+\rho,\alpha\rangle| : \alpha \in \Phi^+\}$ respectively. Then 
$\overline{\mu}=a\approx b$, and $\underline{\mu}=c$. The result follows by applying Proposition \ref{prop: I} to the above bound \eqref{eq: last bound}. 
\end{proof}

\bibliographystyle{siam}
\bibliography{su3}

\end{document}